\documentclass[3p,times]{elsarticle}

\usepackage{amssymb}
\usepackage{amsthm}
\usepackage{amsmath}
\usepackage{amsfonts}
\usepackage{dsfont}
\usepackage{enumitem}
\usepackage{bm}
 \usepackage{lineno}

\usepackage[figuresright]{rotating}
\usepackage{graphicx}

 \newtheorem{thm}{Theorem}[section]

\newtheorem{proposition}[thm]{Proposition}

\newtheorem{e-definition}[thm]{Definition}
\newtheorem{remark}[thm]{\bf Remark\/}

\begin{document}

\begin{frontmatter}
\title{Existence and Uniqueness of  a Transient State for the Coupled Radiative-Conductive Heat Transfer Problem}

\author{Mohamed Ghattassi}\address{University of Lorraine, IECL UMR CNRS 7502, 54506 Vandoeuvre-l\`es-Nancy, France, {\sf mohamed.ghattassi@gmail.com}.}
\author{Jean Rodolphe Roche}\address{University of Lorraine, IECL UMR CNRS 7502, 54506 Vandoeuvre-l\`es-Nancy, France, {\sf  jean-rodolphe.roche@univ-lorraine.fr}.}
\author{Didier Schmitt}\address{University of Lorraine, IECL UMR CNRS 7502, 54506 Vandoeuvre-l\`es-Nancy, France, {\sf didier.schmitt@univ-lorraine.fr}.}

\begin{abstract}
This paper deals with existence and uniqueness results for a transient nonlinear radiative-conductive system in three-dimensional case. This system describes the heat transfer for a grey, semi-transparent and non-scattering medium with general boundary conditions. We reformulate the full transient state system as a fixed-point problem. The existence and uniqueness proof is based on Banach fixed point theorem.
\end{abstract}

\begin{keyword}
Nonlinear radiative-conductive heat transfer system, Semi-transparent medium, existence-uniqueness result,  Banach fixed point theorem.
\end{keyword}

\end{frontmatter}
\section*{Introduction}
The aim of this work is to prove the existence and uniqueness of the solution for a transient combined radiative-conductive system in three-dimensional case with general boundary conditions when the initial condition is assumed to be nonnegative. The medium is assumed grey, semi-transparent and non-scattering.

Let us  consider a convex domain $\Omega\subset\mathbb{R}^{3}$, with $C^{2}$ boundary.  Let ${\bm{\mathcal{S}}}^{2}=\{{\bm{\beta}} \in \mathbb{R}^{3},\,\,\,|{\bm{\beta}}|=1 \}$ be the unit sphere in $\mathbb{R}^{3}$ (the sphere of directions). Let $t \in (0,{\bm{\tau}})$ for $ {\bm{\tau}}>0$,  ${\bm{\mathcal{X}}}=\Omega \times {\bm{\mathcal{S}}}^{2}$, ${\bm{Q_{\tau}}}=(0,{\bm{\tau}}) \times\Omega$ and ${\bm{\Sigma_{{\bm{\tau}}}}}=(0,{\bm{\tau}}) \times \partial\Omega$. Let ${\bf n}$ be the outward unit normal to the boundary $\partial\Omega$.
We denote
\begin{equation*}
\partial\Omega_{-}=\{(x,{\bm{\beta}})\in \partial\Omega\times {\bm{\mathcal{S}}}^{2}~~\mbox{such that}~~{\bm{\beta}}.{\bf n}< 0\}.
\end{equation*}
The full system of a combined nonlinear radiation-conduction heat transfer  is written in dimensionless form,
\begin{align}
\small
&I(t,{\bm{x}},{\bm{\beta}})+{ {\bm{\beta}}}.\nabla_{\bm{x}} I(t,{\bm{x}},{\bm{\beta}})=T^{4}(t,{\bm{x}})\hspace{2.8cm}(t,{\bm{x}},{\bm{\beta}}) \in (0,{\bm{\tau}})\times{\bm{\mathcal{X}}}\label{radiatif2}
\\
&\partial_{t} T(t,{\bm{x}})- \Delta T(t,{\bm{x}}) +4\pi \theta T^{4}(t,{\bm{x}})= \theta G(t,{\bm{x}})\hspace{2.9cm}(t,{\bm{x}}) \in {\bm{Q_{\tau}}}\label{chaleur}\\
&a\partial_{n} T(t,{\bm{x}})+bT(t,{\bm{x}})= g(t,{\bm{x}})\hspace{5.0cm}(t,{\bm{x}}) \in {\bm{\Sigma_{{\bm{\tau}}}}}\label{CBchaleur}\\
&I(t,{\bm{x}},{\bm{\beta}})=I_{b}(t,{\bm{x}},{\bm{\beta}})  \hspace{4.8cm}(t,{\bm{x}},{\bm{\beta}}) \in(0,{\bm{\tau}})\times\partial\Omega_{-}\label{radiativeboundary}\\
&T(0,{\bm{x}})=T_{0}({\bm{x}})\hspace{7.7cm}{\bm{x}}\in\Omega \label{chaleur5}
\end{align}
where $I$ is the dimensionless radiation intensity, $T$ is the dimensionless temperature, $\theta$ is a positive dimensionless constant and $a$ and $b$ are real numbers satisfying the condition of $|a|+|b|>0$. The incident radiation intensity $G$ is given by 
\begin{equation} 
G(t,{\bm{x}})= \int_{{\bm{\mathcal{S}}}^{2}} I(t,{\bm{x}},{\bm{\beta}})d{{\bm{\beta}}}\hspace{1.2cm}(t,{\bm{x}}) \in  {\bm{Q_{\tau}}}.
 \label{moyenne}
\end{equation}
In this paper, we assume that the mean radiation intensity of the grey medium verifies the Stefan-Boltzmann law, which is proportional to $T^4$. The radiative transfer equation (RTE) (\ref{radiatif2}) and the conductive equation (CE) (\ref{chaleur}) are coupled via the source term $\theta\{G-4\pi T^{4}\}$. We use nonhomogeneous Dirichlet boundary $I_{b}$ conditions for radiation equation and different cases of boundary conditions $g$  for CE. For a fuller treatment of the dimensionless form of radiative conductive heat transfer system, we refer the reader to \cite{Ghattassi_20151}.

Radiative-conductive heat transfer problems are the subject of various fields of engineering and science, e.g., glass manufacturing when a hot melt of glass is cooled down to room temperature. Nowadays there is a huge literature on mathematical theory in the radiative-conductive heat transfer problem, see \cite{Amosov_20102, Amosov_2010, Amosov_20101, Amosov_2011, Kovtanyuk_2014, Thompson_2004,Laitinen1997, Laitinen1998, Laitinen2001, Laitinen2002}. For example, the paper \cite{ Amosov_20102} is devoted to the study of a nonstationary, nonlinear, nonlocal initial boundary value problem governing radiative conductive heat transfer in opaque bodies with surfaces whose properties depend on the radiation frequency. This paper is a natural extension of the work done by \cite{Amosov_20101}, where the corresponding stationary problem was treated.  In \cite{ Kovtanyuk_2013}, the authors considered the radiative-conductive heat transfer in a scattering and absorbing medium bounded by two reflecting and radiating plane surfaces. The existence and uniqueness of a solution to this problem are established by using an iterative procedure.

 In \cite{Laitinen2001}, M. Laitinen and T. Tiihonen studied the well-posedness of a class of models describing heat transfer by conduction and radiation in the stationary case. This theory covers different types of grey materials: semitransparent and opaque bodies as well as isotropic or non-isotropic scattering/ reflection. They also revealed that the material properties do not depend on the wavelength of the radiation.
 
In this paper, we consider the coupled system of nonlinear partial differential equations in  three-dimensional case. In previous studies, we found theoretical of existence and uniqueness in  one-dimensional case. Indeed, in the Kelley's paper  \cite{kelley_existence_1996},  the authors considered a steady-state combined radiative-conductive heat transfer. In  Asllanaj et al.\cite{asllanaj_existence_2003} the authors generalized the Kelley's study and they proved the existence and uniqueness of the 1-D system of coupled radiative conductive in the steady state associated to the nonhomogeneous Dirichlet boundary with the black surfaces. The medium is assumed to be a non-grey anisotropic absorbing, emitting, scattering,  with axial symmetry and nonhomogeneous. They considered a nonlinear conduction equation due to the temperature dependence of the thermal conductivity. However, the approach developed by Asllanaj et al. \cite{asllanaj_existence_2003} is just adaptable to 1D dimensional geometry.

We can also find in the literature some results in multidimensional case, see \cite{Amosov_1977,Amosov_2005,Amosov_1980, Amosov1979, Amosov19791,Amosov113, Amosov_2016}. For example in \cite{ Amosov1979, Amosov19791} the authors considered three-dimensional stationary case. Recently, A.A. Amosov \cite{Amosov_2016} proves a result on the unique solvability of a nonstationary problem of radiative-conductive heat transfer in a system of semitransparent bodies (3D case) for the homogeneous conductive boundary conditions. The radiation transfer equation is associated with boundary conditions of mirror reflection and refraction according to the Fresnel laws is used to describe the propagation of radiation.

 Moreover, M.M. Porzio and \'O. L\'opez Pouso proved in  \cite{Pouso1} an existence and uniqueness theorem for the non-grey coupled convection-conduction-radiation system associated to the mixed nonhomogenous Dirichlet and homogenous Neumann boundary conditions by means of accretive operators theory. Leaving aside the grey or non-grey character, the main difference between our problem and the one studied in \cite{Pouso1} is that we do not include the transient term in the RTE. This is an interesting point because this term is really negligible in a wide range of applications; e.g., thermoforming glass see \cite{asllanaj_transient_2007,Ghattassi_20151} and references therein. Moreover, the techniques used in \cite{Pouso1} do not allow disregarding it. In our study, we also discuss different types of boundary conditions.
 
 In this paper, we prove the existence and uniqueness of solution for the nonlinear radiative conductive system in 3-dimensional case associated to the nonhomogeneous Dirichlet boundary conditions for radiation equation and for different types of conductive boundary conditions. The Banach fixed point theorem is the principal tool used to solve this problem.

Recently, some attention has been accorded to numerical methods to study the radiative transfer and the nonlinear radiative-conductive heat transfer problem including optimal control problems, for more details see \cite{Asllanaj_N2001, asllanaj_2004, asllanaj_existence_2003, Lopez_20031, Lopez_2003, Pinnau_2010, Ghattassi_20151, Ghattassi_20151AN, mishra_analysis_2013, Pinnau_2008, Pinnau_2007,Pinnau_2004, Mishra2007,Kovtanyuk_2015, Kovtanyuk_2016, Ghattassisiam} and the pioneering book \cite{consiglieri2011mathematical} and references therein. Asllanaj et al. \cite{asllanaj_transient_2007} simulated transient heat transfer by radiation and conduction in two-dimensional complex shaped domains with structured and unstructured triangular meshes working with an absorbing, emitting and non-scattering grey medium.

The plan of this paper is as follows: Section \ref{sec2}, contains the statement of the main result (theorem \ref{thmintro}). Section \ref{Localexisection} is devoted to its proof based on Banach fixed point theorem.

\section{Main results}\label{sec2}
In order to state the main result, we introduce the following notations
\begin{equation*}
\begin{aligned}
L^{p}({\bm{Q_{\tau}}})&:=L^{p}(0,{\bm{\tau}};L^{p}(\Omega))~~\mbox{for all}~~~p\in [1,\infty[,\\
W_{p}^{2,1}({\bm{Q_{\tau}}})&:=\{\phi~~~\text{such that}~~\phi, \phi_{t},\phi_{x_{i}}, \phi_{x_{i},x_{j}}\in L^{p}({\bm{Q_{\tau}}})\}~\forall~~p\in [1,\infty[.
\end{aligned}
\end{equation*}
According to the method introduced in \cite{cessenat_Theoremes_1984,cessenat_Theoremes_1985, dautrayLions2000} to solve the neutron equations, we consider the following space
\begin{equation*}
{\bm{\mathcal{W}}}^{2}=\{v\in L^{2}({\bm{\mathcal{X}}})~~\text{such that}~{\bm{\beta}}.\nabla_{x} v \in L^{2}( {\bm{\mathcal{X}}})\}
\end{equation*}
and the following subset of  $\partial\Omega \times {\bm{\mathcal{S}}}^{2}$
\begin{equation*}
\partial\Omega_{+}= \{(x,{\bm{\beta}}) \in  \partial\Omega \times {\bm{\mathcal{S}}}^{2}~~\mbox{and}~~{\bm{\beta}}.{\bf n}>0\}.
\end{equation*}
 We denote by 
  $$L^{2}=L^{2}({\bm{\mathcal{X}}}),~L^{2}_{-}=L^{2}(\partial\Omega_{-}; -{\bm{\beta}}.{\bf n}\,\,d{\bm{\Gamma}}d{\bm{\beta}})$$
 and
  $$L^{2}_{+}=L^{2}(\partial\Omega_{+}; {\bm{\beta}}.{\bf n}\,\, d{\bm{\Gamma}} d{\bm{\beta}}),$$ the spaces of square integrable functions  in $ {\bm{\mathcal{X}}}, \partial\Omega _ {-} $ and $ \partial\Omega_{+}$, respectively.
 Let us denote by ${\bm{\mathcal{W}}}$ the following subset of ${\bm{\mathcal{W}}}^{2}$:
\begin{equation}\nonumber
{\bm{\mathcal{W}}}=\{v\in {\bm{\mathcal{W}}}^{2}~~\text{such that}~v_{|_{\partial\Omega_{+}}}\in L^{2}_{+}\}.
\end{equation}
The space  ${\bm{\mathcal{W}}}$ is a Hilbert space when is equipped with the scalar product
\begin{equation*}
(u,v)_{{\bm{\mathcal{W}}}}=\int_{{\bm{\mathcal{X}}}} uv  d{\bm{x}}d{\bm{\beta}} + \int_{{\bm{\mathcal{X}}}}({\bm{\beta}}.\nabla_{{\bm{x}}} u)({\bm{\beta}}.\nabla_{x} v)  d{\bm{x}}d{\bm{\beta}}+ \int_{\partial\Omega_{+}}({\bm{\beta}}.{\bf n}) u v d{\bm{\Gamma}}d{\bm{\beta}}
\end{equation*}
and the  norm
\begin{equation*}
\Big\|u\Big\|_{{\bm{\mathcal{W}}}}=\left(\Big\|u\Big\|^{2}_{L^{2}}+\Big\|{\bm{\beta}}.\nabla_{x} u\Big\|^{2}_{L^{2}}+ \Big\|u\Big\|^{2}_{L^{2}_{+}}\right)^{\frac{1}{2}}.
\end{equation*}
Our result will be obtained under the following assumptions
\begin{equation}
\begin{aligned}
&\ast\,\, I_{b}\in L^{2}(0,{\bm{\tau}};L^{2}_{-})\cap C^{2}(0,\infty;C^{1}(\partial\Omega_{-}))~\text{is nonnegative},\\
&\ast\,\, g\in W_{\infty}^{2,1}((0,\infty)\times\overline{\Omega})\cap C^{2}(0,\infty;C^{1}(\overline{\Omega}))~\text{is nonnegative},\\
&\ast\,\, T _{0}~\text{is nonnegative, belongs to}~H^{1}(\Omega),\\
&\ast\,\, \text{ In the case of Dirichlet boundary conditions, we assume that } {T_{0}}_{|_{\partial \Omega}}=b^{-1}{g(0,.)}_{|_{\partial \Omega}}.
\end{aligned}
\label{assumptions}
\end{equation}
Here and throughout this paper, we shall use $C(...)$ to denote several positive constants depending on what is enclosed in the bracket.
The main result of this paper is the following theorem.
\begin{thm}
Assume that the data verifies  (\ref{assumptions}). Let ${\bm{\tau}}>0$, there exists $\delta=\delta(\Omega, {\bm{\tau}}, \theta)>0$ such that if
$$
\Big\|T_{0}\Big\|_{H^{1}(\Omega)}\leqslant \delta,\,\,\,\,\,\Big\| g \Big\|_{L^{2}(0,{\bm{\tau}};H^{\frac{3}{2}}(\Omega))}\leqslant \delta\,\,\text{and}\,\,\Big\|I_{b}\Big\|_{L^{2}(0,{\bm{\tau}};L^{2}_{-})}\leqslant \delta,
$$
then the problem (\ref{radiatif2})-(\ref{chaleur5}) has a unique nonnegative solution $(T,I)$ such that  $T\in W_{2}^{2,1}({\bm{Q_{\tau}}})$ and  $I \in L^{2}(0,{\bm{\tau}};{\bm{\mathcal{W}}})$. Moreover, there exists $C(\Omega, {\bm{\tau}}, \theta)>0$ such that
\begin{equation}
\Big\|T\Big\|_{W_{2}^{2,1}({\bm{Q_{\tau}}})}\leqslant C(\Omega, {\bm{\tau}}, \theta)\left(\Big\|I_{b} \Big\|_{L^{2}(0,{\bm{\tau}};L^{2}_{-})}+\Big\|T_{0}\Big\|_{H^{1}(\Omega)}+\Big\| g \Big\|_{L^{2}(0,{\bm{\tau}};H^{\frac{3}{2}}(\Omega))} \right).
\label{LYInj1}
\end{equation}
\label{thmintro}
\end{thm}
\begin{remark}
Theorem \ref{thmintro} shows the local existence and uniqueness of the solution for problem (\ref{radiatif2})-(\ref{chaleur5}). Note that if the initial datum is sufficiently small, we have the global existence and uniqueness results.
 \end{remark}
The existence result derives from the application of the Banach fixed point theorem to a suitable map $\mathcal{H}$. The next section is devoted to the construction of this map, composed of three continuous functions. Moreover, since $\mathcal{H}$ is a contraction in a well-chosen set, we deduce the uniqueness of the solution.
\section{Existence and uniqueness of solution for the coupled system}\label{Localexisection}
In this section, we show that the existence of a solution $T$, and implicitly the existence of a solution $I$, of the coupled system of equations $(\ref{radiatif2})$-$(\ref{chaleur5})$ is related to the existence of a solution of a  fixed point problem. We will apply the fixed point theorem to a well-chosen map $\mathcal{H}$. To do so, we must show that this map $\mathcal{H}$ has a unique solution $T$.

The map $\mathcal{H}: W_{2}^{2,1}({\bm{Q_{\tau}}}) \longrightarrow W_{2}^{2,1}({\bm{Q_{\tau}}})$ is a composition of three maps
\begin{equation*}
\mathcal{H}=\mathcal{H}_{3}\circ \mathcal{H}_{2} \circ \mathcal{H}_{1}.
\end{equation*}

The map $\mathcal{H}_{1}$: $W_{2}^{2,1}({\bm{Q_{\tau}}}) \longrightarrow L^{2}({\bm{Q_{\tau}}})$ is defined as follows, for  $T \in W_{2}^{2,1}({\bm{Q_{\tau}}})$,  $T^{4}=\mathcal{H}_{1}(T) \in L^{2}({\bm{Q_{\tau}}})$. On the other hand, the map $\mathcal{H}_{2}:L^{2}({\bm{Q_{\tau}}}) \longrightarrow L^{2}({\bm{Q_{\tau}}})$ is defined as follows, for  $T^{4}\in L^{2}({\bm{Q_{\tau}}})$,  $I=\mathcal{H}_{2}(T^{4}) \in L^{2}({\bm{Q_{\tau}}})$. Finally, the map $\mathcal{H}_{3}: L^{2}({\bm{Q_{\tau}}}) \longrightarrow W_{2}^{2,1}({\bm{Q_{\tau}}})$  is defined as follows, for $G \in L^{2}({\bm{Q_{\tau}}}) $,  $\mathcal{H}_{3}(G) \in W_{2}^{2,1}({\bm{Q_{\tau}}})$ is the solution of CE \eqref{chaleur},\eqref{CBchaleur} and \eqref{chaleur5}.

To study $\mathcal{H}$, we will be studying in great detail the maps $\mathcal{H}_{1}$, $\mathcal{H}_{2}$ and $\mathcal{H}_{3}$. 
\subsection{The maps $\mathcal{H}_{1}$ and $\mathcal{H}_{2}$}
Now, we focus on the maps $\mathcal{H}_{1}$ and $\mathcal{H}_{2}$, we give some properties of the solution of the RTE  (\ref{radiatif2}) using nonhomogeneous radiative Dirichlet boundary conditions.

We start by recalling the Green's formula, see \cite{dautrayLions2000}:
\begin{equation}
\int_{{\bm{\mathcal{X}}}} ({\bm{\beta}}.\nabla_{{\bm{x}}} u)v d{\bm{x}}d{\bm{\beta}}+ \int_{{\bm{\mathcal{X}}}}({\bm{\beta}}.\nabla_{{\bm{x}}} v) u d{\bm{x}} d{\bm{\beta}} = \int_{\partial\Omega \times  {\bm{\mathcal{S}}}^{2}}({\bm{\beta}}.{\bf n}) u v d{\bm{\Gamma}}d{\bm{\beta}},
\label{formulegreen}
\end{equation}
for all $(u,v)\in {\bm{\mathcal{W}}}\times{\bm{\mathcal{W}}}$.
\begin{thm}
Let us consider $T\in W_{2}^{2,1}({\bm{Q_{\tau}}})$. Under the assumptions (\ref{assumptions}), the problem (\ref{radiatif2}), (\ref{radiativeboundary}) has a unique nonnegative solution $I\in L^{2}(0,{\bm{\tau}};{\bm{\mathcal{W}}})$. Moreover, there exists $C_{1}=C(\tau, \Omega)>0$ such that 
\begin{equation}
\begin{aligned}
\Big\|I\Big\|_{L^{2}(0,{\bm{\tau}};{\bm{\mathcal{W}}})}\leqslant C_{1}\left(\Big\|T\Big\|^{4}_{W_{2}^{2,1}({\bm{Q_{\tau}}})}+\Big\|I_{b}\Big\|_{L^{2}(0,{\bm{\tau}};L^{2}_{-})}\right).
\end{aligned}
\label{GL22}
\end{equation}
\label{thsolrad}
\end{thm}
 \begin{proof}
Let   $T\in W_{2}^{2,1}({\bm{Q_{\tau}}})$, $t\in [0,{\bm{\tau}}]$, we have $T^{4}(t) \in L^{2}(\Omega)$. Using the existence and uniqueness of the solution for the transport equation, see \cite{dautrayLions2000}, the boundary value problem $(\ref{radiatif2}),(\ref{radiativeboundary})$ has a unique solution $I(t)\in L^{2}({\bm{\mathcal{X}}})$.

In addition, using the linearity of $(\ref{radiatif2})$, the solution $I$ of the problem $(\ref{radiatif2}),(\ref{radiativeboundary})$ is given by $I = I_{0} + w$ where $I_ {0}$ is a solution of $(\ref{radiatif2})$, $(\ref{radiativeboundary})$  for $I_{b}\equiv 0$  and $ w $ is a solution of $(\ref{radiatif2})$, $(\ref{radiativeboundary})$ without the second member $T^{4}$.

 We start by the homogeneous problem
\begin{align}
{\bm{\beta}}.\nabla_{x} I_{0}(t,{\bm{x}},{\bm{\beta}}) + I_{0}(t,{\bm{x}},{\bm{\beta}})&=T^{4}(t,{\bm{x}})~~~~\forall (t,{\bm{x}},{\bm{\beta}}) \in (0,{\bm{\tau}})\times{\bm{\mathcal{X}}}\label{equadem2}\\
I_{0}(t,{\bm{x}},{\bm{\beta}})&=0 ~~~~~~~~~~~~~~~\forall (t,{\bm{x}},{\bm{\beta}})\in \partial\Omega_{-}\label{equadem2FF}.
\end{align}
If we multiply the equation $(\ref{equadem2})$ by $I_{0}$, we integrate in space and in direction and we use Green's formula ($\ref{formulegreen}$) and (\ref{equadem2FF}) to get

\begin{equation}
\begin{aligned}
\Big\|I_{0}(t)\Big\|_{L^{2}} \leqslant \sqrt{4\pi} \Big\|T^{4}(t)\Big\|_{L^{2}(\Omega)}.
\label{L21}
\end{aligned}
\end{equation}
If we multiply $(\ref{equadem2})$ by $I_{0}+{\bm{\beta}}.\nabla_{x} I_{0}$ and we integrate in  space and in direction, we obtain

\begin{equation}
\Big\|I_{0}(t)\Big\|_{{\bm{\mathcal{W}}}} \leqslant \sqrt{4\pi}\Big\|T^{4}(t)\Big\|_{L^{2}(\Omega)}.
\label{equaestim1}
\end{equation} 
Now, we study the nonhomogeneous boundary value problem:
\begin{align}
{\bm{\beta}}.\nabla_{x} w(t,{\bm{x}},{\bm{\beta}}) + w(t,{\bm{x}},{\bm{\beta}})&=0~~~~~~~~~~~~~~\forall (t,{\bm{x}},{\bm{\beta}}) \in (0,{\bm{\tau}})\times{\bm{\mathcal{X}}}\label{equa444}\\
w(t,{\bm{x}},{\bm{\beta}})&=I_{b}(t,{\bm{x}},{\bm{\beta}})~~~\forall (t,{\bm{x}},{\bm{\beta}}) \in (0,{\bm{\tau}})\times\partial\Omega_{-}.
\end{align}
Multiplying by  $ w $ and  integrating in ${\bm{\mathcal{X}}}$, we find that 
\begin{equation}
\Big\|w(t)\Big\|_{L^{2}}\leqslant \frac{1}{\sqrt{2}}\Big\|I_{b}(t)\Big\|_{L^{2}_{-}}.
\label{L22}
\end{equation}
If we multiply $(\ref{equa444})$ by ${\bm{\beta}}.\nabla_{x} w$ and we integrate in ${\bm{\mathcal{X}}}$, we obtain
\begin{equation}
\Big\|w(t)\Big\|_{{\bm{\mathcal{W}}}}\leqslant\Big\|I_{b}(t)\Big\|_{L^{2}_{-}}.
\label{equaestim11}
\end{equation}
Since $I=I_{0}+w$, the estimates (\ref{L21})  and (\ref{L22}) imply  
\begin{equation*}
\begin{aligned}
\Big\|I(t)\Big\|_{L^{2}}&\leqslant \sqrt{4\pi}\Big\|T^{4}(t)\Big\|_{L^{2}(\Omega)}+\frac{1}{\sqrt{2}}\Big\|I_{b}(t)\Big\|_{L^{2}_{-}}.
\end{aligned}
\end{equation*}
Finally, in a similar way, according to (\ref{equaestim1}) and (\ref{equaestim11}), we obtain
\begin{equation*}
\begin{aligned}
\Big\|I(t)\Big\|_{{\bm{\mathcal{W}}}}&\leqslant \sqrt{4\pi}\Big\|T^{4}(t)\Big\|_{L^{2}(\Omega)}+\Big\|I_{b}(t)\Big\|_{L^{2}_{-}}.
\end{aligned}
\end{equation*}
If we integrate in time  between  $0$ and ${\bm{\tau}}$, we obtain
\begin{equation*}
\begin{aligned}
&\Big\|I\Big\|_{L^{2}(0,{\bm{\tau}};L^{2})}\leqslant \sqrt{4\pi}\Big\|T\Big\|^{4}_{L^{8}({\bm{Q_{\tau}}})}+\frac{1}{\sqrt{2}}\Big\|I_{b}\Big\|_{L^{2}(0,{\bm{\tau}};L^{2}_{-})}\\
&\Big\|I\Big\|_{L^{2}(0,{\bm{\tau}};{\bm{\mathcal{W}}})}\leqslant \sqrt{4\pi}\Big\|T\Big\|^{4}_{L^{8}({\bm{Q_{\tau}}})}+\Big\|I_{b}\Big\|_{L^{2}(0,{\bm{\tau}};L^{2}_{-})}.
\end{aligned}
\end{equation*}
Then, using the continuous embedding $W_{2}^{2,1}({\bm{Q_{\tau}}})\hookrightarrow L^{8}({\bm{Q_{\tau}}})$, see \cite{ladyzenskaja_linear_1968}, there exists $C(\tau, \Omega)>0$ such that 
\begin{equation*}
\begin{aligned}
&\Big\|I\Big\|_{L^{2}(0,{\bm{\tau}};L^{2})}\leqslant C(\tau, \Omega)\left(\Big\|T\Big\|^{4}_{W_{2}^{2,1}({\bm{Q_{\tau}}})}+\Big\|I_{b}\Big\|_{L^{2}(0,{\bm{\tau}};L^{2}_{-})}\right),\\
&\Big\|I\Big\|_{L^{2}(0,{\bm{\tau}};{\bm{\mathcal{W}}})}\leqslant C(\tau, \Omega)\left(\Big\|T\Big\|^{4}_{W_{2}^{2,1}({\bm{Q_{\tau}}})}+\Big\|I_{b}\Big\|_{L^{2}(0,{\bm{\tau}};L^{2}_{-})}\right).
\end{aligned}
\end{equation*}
Using the positivity of $I_{b}$ and the maximum principle \cite{agoshkov_boundary_1998}, this implies that the solution $I$ of  $(\ref{radiatif2}), (\ref{radiativeboundary})$ is nonnegative.
 \end{proof}

\begin{proposition}
 Under the hypotheses of theorem \ref{thsolrad}, the map $\mathcal{H}_{2}o\mathcal{H}_{1}$ is a well-posed and continuous map  from $W_{2}^{2,1}({\bm{Q_{\tau}}})$ to $L^{2}({\bm{Q_{\tau}}})$. 
\end{proposition}
\begin{proof}
Let $T\in W_{2}^{2,1}({\bm{Q_{\tau}}})$, theorem \ref{thsolrad} implies that  the map $\mathcal{H}_{2}$ is a well-posed and continuous map  from $W_{2}^{2,1}({\bm{Q_{\tau}}})$ to $L^{2}({\bm{Q_{\tau}}})$. 

The map $\mathcal{H}_{2}$ is defined from $L^{2}({\bm{Q_{\tau}}})$ to $L^{2}({\bm{Q_{\tau}}})$ by
$$
\mathcal{H}_{2}(T^{4})=G
$$
where $G$ is given by \eqref{moyenne}.

From theorem \ref{thsolrad} and using (\ref{moyenne}) we can deduce that $\mathcal{H}_{2}$ is a well-posed and continuous map  from $L^{2}({\bm{Q_{\tau}}})$ to $L^{2}({\bm{Q_{\tau}}})$. 
Moreover, there exists $C_{2}=C(\tau, \Omega)>0$ such that 
\begin{equation}
\Big\|G\Big\|_{L^{2}({\bm{Q_{\tau}}})}\leqslant C_{2}\left(\Big\|T\Big\|^{4}_{W_{2}^{2,1}({\bm{Q_{\tau}}})}+\Big\|I_{b}\Big\|_{L^{2}(0,{\bm{\tau}};L^{2}_{-})}\right).
\label{G}
\end{equation}
 Finally, it follows that $\mathcal{H}_{2}o\mathcal{H}_{1}$ is a well-posed and continuous  map  from $W_{2}^{2,1}({\bm{Q_{\tau}}})$ to $L^{2}({\bm{Q_{\tau}}})$.
\end{proof}

\subsection{The map $\mathcal{H}_{3}$}
In this subsection  we start by introducing some properties  of the map $\mathcal{H}_{3}$.
\begin{proposition}
Let ${\bm{\tau}}>0$, $G \in L^{2}({\bm{Q_{\tau}}})$. Under the assumptions (\ref{assumptions}), the problem (\ref{chaleur}),(\ref{CBchaleur}),(\ref{chaleur5}) has a nonnegative solution $T\in W_{2}^{2,1}({\bm{Q_{\tau}}})$. Moreover, there exists $C_{3}=C(\Omega, {\bm{\tau}}, \theta)>0$ such that
\begin{equation}
\Big\|T\Big\|_{W_{2}^{2,1}({\bm{Q_{\tau}}})}\leqslant C_{3}\Big(\Big\|G\Big\|_{L^{2}({\bm{Q_{\tau}}})}+\Big\|T_{0}\Big\|_{H^{1}(\Omega)}+\Big\| g \Big\|_{L^{2}(0,{\bm{\tau}};H^{\frac{3}{2}}(\Omega))} \Big)
\label{C3}
\end{equation}
 and $\mathcal{H}_{3}$ is  a continuous map from  $L^{2}({\bm{Q_{\tau}}})$ to  $W_{2}^{2,1}({\bm{Q_{\tau}}})$.
\label{propchalex}
\end{proposition}

\begin{proof}  Let ${\bm{\tau}}>0$, $T_{0}\in H^{1}(\Omega)$ and  $G\in L^{2}(\Omega)$. The existence and uniqueness of solution for the problem (\ref{chaleur}),(\ref{CBchaleur}) and (\ref{chaleur5}), is discussed in \cite[Chapter 5]{Cazenave_1998}.

Now, in order to prove the non-negativity  of the solution for (\ref{chaleur}),(\ref{CBchaleur}) and (\ref{chaleur5}), let us consider  $F$ defined in  $(0,{\bm{\tau}}) \times\Omega \times \mathbb{R}$ by
\begin{equation}\label{F}
F(t,{\bm{x}},y)=\theta\left ( G(t, {\bm{x}})-4\pi y^{4}\right ).
\end{equation}
The system (\ref{chaleur}),(\ref{CBchaleur}) and (\ref{chaleur5}) can be rewritten
\begin{equation}\label{1}
 \left\lbrace
\begin{aligned}
&\partial_{t}T(t, {\bm{x}})-\Delta T(t, {\bm{x}})=F(t,{\bm{x}}, T(t, {\bm{x}}))~~~~~\hbox{ for}~~(t,{\bm{x}}) \in ]0,{\bm{\tau}}] \times\Omega \\
&a\partial_{n} T(t,{\bm{x}})+b T(t,{\bm{x}})= g(t,{\bm{x}}) ~~~~~~~~~~~~~~~\hbox{ for}~(t,{\bm{x}}) \in ]0,{\bm{\tau}}] \times \partial\Omega\\
&T(0, {\bm{x}})=T_{0}({\bm{x}})\quad~~~~~~~~~~~~~~~~~~~~~~~~~~~~~~~\hbox{ for}~{\bm{x}} \in \Omega.\\
\end{aligned}
\right.
\end{equation}

Now, we define $\bar{F}$ in $(0, {\bm{\tau}}) \times\Omega \times \mathbb{R}$ by
\[
\bar{F}(t, {\bm{x}},y)=\left \{\begin{array}{lll}
\theta\left ( G(t,{\bm{x}})-4\pi y^{4}\right )&{\rm if}& y\geq 0\\
\theta G(t, {\bm{x}})&{\rm if}& y<0.
\end{array}
\right.
\]
Let us consider  $\bar{T}$ the solution of the following system
\begin{equation}\label{2}
 \left\lbrace
\begin{aligned}
&\partial_{t}\bar{T}(t, {\bm{x}})-\Delta \bar{T}(t, {\bm{x}})=\bar{F}(t,{\bm{x}}, \bar{T}(t, {\bm{x}}))~~~~~\hbox{ for}~~(t,{\bm{x}}) \in ]0,{\bm{\tau}}] \times\Omega \\
&a\partial_{n} \bar{T}(t,{\bm{x}})+b\bar{T}(t,{\bm{x}})= g(t,{\bm{x}}) ~~~~~~~~~~~~~~~\hbox{ for}~(t,{\bm{x}}) \in ]0,{\bm{\tau}}] \times \partial\Omega\\
&\bar{T}(0, {\bm{x}})=T_{0}({\bm{x}})\quad~~~~~~~~~~~~~~~~~~~~~~~~~~~~~~~\hbox{ for}~{\bm{x}} \in \Omega.\\
\end{aligned}
\right.
\end{equation}

\noindent  Our goal is to prove that the solution $\bar{T}$ of this equation remains nonnegative over the time. Indeed, in this case  $\bar{F}$ and $F$ coincide, therefore we have by the uniqueness of the solution $T = \bar{T}$ which is  nonnegative.

\noindent  We set $\bar{T}^{+}=\max(T,0)$ and $\bar{T}^{-}=\max(-T,0)$, such that $\bar{T}=\bar{T}^{+}-\bar{T}^{-}$.

 Multiplying  the equation (\ref{2}) by $(-\bar{T}^{-})$ and integrating over $\Omega$,  we obtain
\begin{equation*}
-\int_{\Omega}\partial_{t}\bar{T}(t, {\bm{x}})\bar{T}^{-}(t, {\bm{x}})d{\bm{x}}+\int_{\Omega}\Delta \bar{T}(t, {\bm{x}})\bar{T}^{-}(t, {\bm{x}})d{\bm{x}}= -\int_{\Omega}\bar{F}(t, {\bm{x}},\bar{T}) \bar{T}^{-}(t, {\bm{x}})d{\bm{x}}.
\end{equation*}
Now, we have
\begin{equation}\label{derivetemp}
-\int_{\Omega}\partial_{t}\bar{T}(t, {\bm{x}})\bar{T}^{-}(t, {\bm{x}})d{\bm{x}}= \frac{1}{2}\partial_{t}\int_{\Omega}(\bar{T}^{-}(t, {\bm{x}}))^{2}d{\bm{x}},
\end{equation}
\begin{equation}\label{F}
\begin{aligned}
-\int_{\Omega}\bar{F}(t, {\bm{x}},\bar{T}) \bar{T}^{-}(t, {\bm{x}})d{\bm{x}}&=-\int_{\{\bar{T}<0\}}\bar{F}(t, {\bm{x}},\bar{T}) \bar{T}^{-}(t, {\bm{x}})d{\bm{x}}\\&=-\theta\int_{\{\bar{T}<0\}}G(t, {\bm{x}}) \bar{T}^{-}(t, {\bm{x}})d{\bm{x}}\leq 0,
\end{aligned}
\end{equation}
and 
\begin{equation*}
\begin{aligned}
\int_{\Omega}\Delta \bar{T}(t, {\bm{x}})\bar{T}^{-}(t, {\bm{x}})d{\bm{x}}=&\int_{\Omega}(\nabla \bar{T}^{-}(t, {\bm{x}}))^{2}d{\bm{x}}+\int_{\partial\Omega}\partial_{n}\bar{T}(t, {\bm{x}})\bar{T}^{-}(t, {\bm{x}})d{\bm{\Gamma}}.
\end{aligned}
\end{equation*}
If $a\neq0$ (Robin or Neumann boundary conditions), then
\begin{equation}
\begin{aligned}
\int_{\partial\Omega}\partial_{n}\bar{T}(t, {\bm{x}})\bar{T}^{-}(t, {\bm{x}})d{\bm{\Gamma}}=&-\frac{b}{a}\int_{\partial\Omega}\bar{T}(t, {\bm{x}})\bar{T}^{-}(t, {\bm{x}})d{\bm{\Gamma}} \\&+\frac{1}{a}\int_{\partial\Omega}g(t, {\bm{x}})\bar{T}^{-}(t, {\bm{x}})d{\bm{\Gamma}}\\=&\frac{b}{a}\int_{\partial\Omega}\left(\bar{T}^{-}(t, {\bm{x}})\right)^{2}d{\bm{\Gamma}} \\&+\frac{1}{a}\int_{\partial\Omega}g(t, {\bm{x}})\bar{T}^{-}(t, {\bm{x}})d{\bm{\Gamma}} \geqslant 0.
\end{aligned}
\label{Robnew}
\end{equation}
Now, if we have $a\neq0$ (thus $b>0$), since  $\bar{T}^{-}=0$ on $\partial \Omega$ then
\begin{equation}
\begin{aligned}
\int_{\partial\Omega}\partial_{n}\bar{T}(t, {\bm{x}})\bar{T}^{-}(t, {\bm{x}})d{\bm{\Gamma}}=0.
\end{aligned}
\label{Dirch}
\end{equation}
In the both cases, we have 
\begin{equation}
\int_{\Omega}\Delta \bar{T}(t, {\bm{x}})\bar{T}^{-}(t, {\bm{x}})d{\bm{x}}\geqslant 0.
\label{Lap}
\end{equation}
Consequently, \eqref{derivetemp}, \eqref{F} and \eqref{Lap} imply
\begin{equation}
\frac{1}{2}\partial_{t}\int_{\Omega}(\bar{T}^{-}(t, {\bm{x}}))^{2}d{\bm{x}}\leq 0.
\label{eqpos}
\end{equation}
As $T_{0}$ is nonnegative, we deduce from \eqref{eqpos} that $\bar{T}^{-}\equiv 0$. It follows that $\bar{T}$ and consequently $T$ are nonnegative in $(0,{\bm{\tau}})\times\Omega$.

In the following, we prove that $T\in W_{2}^{2,1}({\bm{Q_{\tau}}})$. For it, let us introduce $z$ the solution of the parabolic problem
\begin{equation}
 \left\lbrace
\begin{aligned}
\partial_{t}z(t,{\bm{x}})-\Delta z(t,{\bm{x}})&=\theta G(t,{\bm{x}})~~~~~\hbox{ for}~~(t,{\bm{x}}) \in ]0,{\bm{\tau}}] \times\Omega \\
a\partial_{n} z(t,{\bm{x}})+bz(t,{\bm{x}})&= g(t,{\bm{x}})~~~~~~~\hbox{ for}~(t,{\bm{x}}) \in ]0,{\bm{\tau}}] \times \partial\Omega\\
z(t,{\bm{x}})&=T_{0}\quad~~~~~~~~~\hbox{ for}~{\bm{x}} \in \Omega.\\
\end{aligned}
\right.
\label{Aux45}
\end{equation}
Since  $G \in L^2({\bm{Q_{\tau}}})$, $T_0\in  H^{1}(\Omega)$ and thanks to a result on parabolic regularity, see \cite{ladyzenskaja_linear_1968}, then $z \in W_{2}^{2,1}({\bm{Q_{\tau}}})$ and there exists a constant $\tilde{C}=C(\Omega, {\bm{\tau}}, \theta)>0$ such that
\begin{equation}
\Big\|z\Big\|_{W_{2}^{2,1}({\bm{Q_{\tau}}})}\leqslant \tilde{C}\Big(\Big\|G\Big\|_{L^{2}({\bm{Q_{\tau}}})}+\Big\|T_{0}\Big\|_{H^{1}(\Omega)}+\Big\|g\Big\|_{L^{2}(0,{\bm{\tau}};H^{\frac{3}{2}}(\Omega))}  \Big).
\label{LYInj1}
\end{equation}
For more details, we refer the reader to \cite[p.197]{Denk2007}. Then, using the Sobolev embedding we deduce that $T$ is a solution of \eqref{1}, then using the  maximum principle, we have that $T \leq z$. Consequently, $T$ belongs to $L^{8}({\bm{Q_{\tau}}})$.
Therefore, using the compact embedding $W_{2}^{2,1}({\bm{Q_{\tau}}}\hookrightarrow L^{8}({\bm{Q_{\tau}}})$ we deduce that $T\in W_{2}^{2,1}({\bm{Q_{\tau}}})$. Finally, from  \eqref{assumptions} and \eqref{LYInj1} we can deduce that the map $\mathcal{H}_{3}$ is  a  well-posed and continuous  from  $L^{2}({\bm{Q_{\tau}}})$ to  $W_{2}^{2,1}({\bm{Q_{\tau}}})$.
 \end{proof}

Now, we will give some estimates of the solution $T$ for (\ref{chaleur}),(\ref{CBchaleur}) and (\ref{chaleur5}) in $L^{8}({\bm{Q_{\tau}}})$.

\begin{proposition}
Under the hypotheses of Proposition \ref{propchalex}, $T$ satisfies
\begin{equation*}
\Big\|T\Big\|^{8}_{L^{8}({\bm{Q_{\tau}}})} \leqslant \frac{1}{16\pi^{2}}\Big\|G\Big\|^{2}_{L^{2}({\bm{Q_{\tau}}})}+C(a,b,\theta){\bm{\tau}}\Big|\partial\Omega\Big|\Big\|g\Big\|^{5}_{L^{\infty}((0,{\bm{\tau}})\times\overline{\Omega})}+\frac{1}{10\pi \theta}\Big\|T_{0}\Big\|^{5}_{L^{5}(\Omega)},
\end{equation*}
for Robin boundary conditions  ${(a\neq0, b\neq0)}$, 
\begin{equation*}
\begin{aligned}
 \Big\|T\Big\|^{8}_{L^{8}({\bm{Q_{\tau}}})} \leqslant& \frac{5}{64\pi^{2}}\Big\|G\Big\|^{2}_{L^{2}({\bm{Q_{\tau}}})}+{\bm{\tau}}\left(C(\theta)\Big|\Omega\Big|+C(a,\Omega,\theta)\Big\|g\Big\|^{5}_{L^{\infty}((0,{\bm{\tau}})\times\overline{\Omega})}\right)\\&+\frac{1}{8\pi \theta}\Big\|T_{0}\Big\|^{5}_{L^{5}(\Omega)},
\end{aligned}
\end{equation*}
for Neumann boundary conditions ${(a\neq0, b=0)}$ and 
\begin{equation*}
\begin{aligned}
\Big\|T\Big\|^{8}_{L^{8}({\bm{Q_{\tau}}})} \leqslant&\frac{1}{8\pi^{2}}\Big\|G\Big\|^{2}_{L^{2}({\bm{Q_{\tau}}})}+C(b) {\bm{\tau}}\Big|\Omega\Big|\Big\|g\Big\|^{8}_{L^{\infty}({\bm{Q_{\tau}}})}+\frac{1}{5\pi \theta}\Big\|T_{0}\Big\|^{5}_{L^{5}(\Omega)}\\&+C(\theta,b)\Big|\Omega\Big|\Big\|g\Big\|^{5}_{L^{\infty}({\bm{Q_{\tau}}})}+C(\theta,b){\bm{\tau}}\Big|\partial\Omega\Big| \Big\|g\partial_{n}g^{4}\Big\|_{L^{\infty}((0,{\bm{\tau}})\times\overline{\Omega})}\\&+C(\theta,b){\bm{\tau}}\Big|\Omega\Big|\Big\|\left(\partial_{t}-\Delta\right)g^{4}\Big\|^{8/7 }_{L^{\infty}({\bm{Q_{\tau}}})},
\end{aligned}
\end{equation*}
for Dirichlet boundary conditions ${(a\neq0, b>0)}$.
\label{thmchalex}
\end{proposition}

 \begin{proof}
As $T\in W_{2}^{2,1}({\bm{Q_{\tau}}})$ then $T^4$ belongs to $L^{2}(0,{\bm{\tau}} ;H^{1}(\Omega))$. Thus, we can multiply the equation  (\ref{chaleur})  by  $T^{4}$ and we integrate over $\Omega$, we obtain for all $t\in(0,{\bm{\tau}})$ 

\begin{equation}
\begin{aligned}
 \frac{1}{5} \frac{d}{dt} \Big\|T(t )\Big\|^{5}_{L^{5}(\Omega)} +&4\int_{\Omega}(\nabla T(t,{\bm{x}}))^{2}T^{3}(t,{\bm{x}})d{\bm{x}}-\int_{\partial\Omega}\partial_{n} T(t,{\bm{x}})T^{4}(t,{\bm{x}})d{\bm{\Gamma}}\\&+4\pi \theta \int_{\Omega}T^{8}(t,{\bm{x}}) d{\bm{x}}= \theta\int_{\Omega}G(t,{\bm{x}})T^{4}(t,{\bm{x}})d{\bm{x}}.
\end{aligned}
\end{equation}
Using the Young's inequality, for all  $\epsilon>0$ such that 
\begin{equation*}
\begin{aligned}
 \frac{1}{5} \frac{d}{dt} &\Big\|T(t )\Big\|^{5}_{L^{5}(\Omega)}+\frac{16}{25}\int_{\Omega}(\nabla T^{\frac{5}{2}}(t,{\bm{x}}))^{2}d{\bm{x}}+4\pi \theta \int_{\Omega}T^{8}(t,{\bm{x}})d{\bm{x}}\\&\leqslant \frac{\theta \epsilon}{2}\int_{\Omega}G^{2}(t,{\bm{x}})d{\bm{x}}+\frac{\theta}{2\epsilon} \int_{\Omega}T^{8}(t,{\bm{x}})d{\bm{x}}+\int_{\partial\Omega}\partial_{n} T(t,{\bm{x}})T^{4}(t,{\bm{x}})d{\bm{\Gamma}}.
\end{aligned}
\end{equation*}
Choosing ${\epsilon=\frac{1}{4\pi}}$we get 
\begin{equation}
\begin{aligned}
 \frac{1}{5} \frac{d}{dt} \Big\|T(t )\Big\|^{5}_{L^{5}(\Omega)} +&\frac{16}{25}\int_{\Omega}(\nabla T^{\frac{5}{2}}(t,{\bm{x}}))^{2}d{\bm{x}}+2\pi \theta \int_{\Omega}T^{8}(t,{\bm{x}})d{\bm{x}}\\&\leqslant \frac{\theta}{8\pi}\int_{\Omega}G^{2}(t,{\bm{x}})d{\bm{x}}+\int_{\partial\Omega}\partial_{n} T(t,{\bm{x}})T^{4}(t,{\bm{x}})d{\bm{\Gamma}}.
\end{aligned}
\label{CBgeneral}
\end{equation}

The treatment of the boundary terms will be different. We start with the simplest case Robin boundary conditions $(a\neq0, b\neq0)$.
 
Using Young's inequality, we have
\begin{equation*}
\begin{aligned}
\int_{\partial\Omega}\partial_{n} T(t)T^{4}(t,{\bm{x}})d{\bm{\Gamma}}=&-\frac{b}{a}\int_{\partial\Omega}T^{5}(t,{\bm{x}})d{\bm{\Gamma}}+\frac{1}{a}\int_{\partial\Omega}gT^{4}(t,{\bm{x}})d{\bm{\Gamma}}\\\leqslant& -\frac{b}{a}\int_{\partial\Omega}T^{5}(t,{\bm{x}})d{\bm{\Gamma}}+\frac{4\epsilon_{1}}{5a}\int_{\partial\Omega}T^{5}(t,{\bm{x}})d{\bm{\Gamma}}\\&+\frac{C(\epsilon_{1})}{5a}\int_{\partial\Omega}g^{5}(t,{\bm{x}})d{\bm{\Gamma}}.
\end{aligned}
\end{equation*}
Choosing ${\displaystyle \epsilon_{1}=\frac{5b}{4}}$, then there exists $C(a,b)>0$
\begin{equation*}
\begin{aligned}
\int_{\partial\Omega}\partial_{n} T(t)T^{4}(t,{\bm{x}})d{\bm{\Gamma}}\leqslant&C(a,b)\int_{\partial\Omega}g^{5}(t,{\bm{x}})d{\bm{\Gamma}}.
\end{aligned}
\end{equation*}
Thus from \eqref{CBgeneral}, it follows that 
\begin{equation*}
\begin{aligned}
 \frac{1}{5} \frac{d}{dt} \Big\|T(t )\Big\|^{5}_{L^{5}(\Omega)} +2\pi \theta \int_{\Omega}T^{8}(t,{\bm{x}}) d{\bm{x}}\leqslant &\frac{\theta}{8\pi}\int_{\Omega}G^{2}(t,{\bm{x}})d{\bm{x}}\\+C(a,b)\int_{\partial\Omega}g^{5}(t,{\bm{x}})d{\bm{\Gamma}}.
\end{aligned}
\end{equation*}
We integrate in time between $0$ and ${\bm{\tau}}$, we obtain 
\begin{equation*}
\begin{aligned}
\Big\|T\Big\|^{8}_{L^{8}({\bm{Q_{\tau}}})} \leqslant \frac{1}{16\pi^{2}}\Big\|G\Big\|^{2}_{L^{2}({\bm{Q_{\tau}}})}+C(a,b,\theta){\bm{\tau}}\Big|\partial\Omega\Big|\Big\|g\Big\|^{5}_{L^{\infty}((0,{\bm{\tau}})\times\overline{\Omega})}+\frac{1}{10\pi \theta}\Big\|T_{0}\Big\|^{5}_{L^{5}(\Omega)}.
\end{aligned}
\label{estRo}
\end{equation*}
Now, we consider the Neumann boundary conditions ${(a\neq0, b=0)}$. Let consider the boundary term of (\ref{CBgeneral})
\begin{equation*}
\begin{aligned} 
 \int_{\partial\Omega}\partial_{n} T(t,{\bm{x}})T^{4}(t,{\bm{x}})d{\bm{\Gamma}}&=\frac{1}{a}\int_{\partial\Omega}g(t,{\bm{x}})T^{4}(t,{\bm{x}})d{\bm{\Gamma}}\\
 &\leqslant \frac{C(\epsilon)}{5a}\int_{\partial\Omega}g^{5}(t,{\bm{x}})d{\bm{\Gamma}}+\frac{4\epsilon}{5a}\int_{\partial\Omega}T^{5}(t,{\bm{x}})d{\bm{\Gamma}}\\
 &\leqslant \frac{C(\epsilon)}{5a}\int_{\partial\Omega}g^{5}(t,{\bm{x}})d{\bm{\Gamma}}+\frac{4\epsilon}{5a}\int_{\partial\Omega}\left(T^{\frac{5}{2}}(t,{\bm{x}})\right)^{2}d{\bm{\Gamma}}.
\end{aligned}
\end{equation*}
Then there exists $C(\Omega)>0$, see \cite{brezis}, such that 
\begin{equation}
\begin{aligned} 
 \int_{\partial\Omega}\partial_{n} T(t,{\bm{x}})T^{4}(t,{\bm{x}})d{\bm{\Gamma}}&\leqslant \frac{C(\epsilon)}{5a}\int_{\partial\Omega}g^{5}(t,{\bm{x}})d{\bm{\Gamma}}+\frac{4\epsilon}{5a}C(\Omega)\Big\|T^{\frac{5}{2}}(t)\Big\|^{2}_{H^{1}(\Omega)}.
\end{aligned}
\label{NewmanBC}
\end{equation}
Choosing ${\epsilon=\frac{4a}{5C(\Omega)}}$, substituting (\ref{NewmanBC}) into (\ref{CBgeneral}), for all $t\in(0,{\bm{\tau}})$
\begin{equation}
\begin{aligned}
 \frac{d}{dt} \Big\|T(t )\Big\|^{5}_{L^{5}(\Omega)}+&10\pi \theta \int_{\Omega}T^{8}(t,{\bm{x}})d{\bm{x}}\leqslant \frac{5\theta}{8\pi}\int_{\Omega}G^{2}(t,{\bm{x}})d{\bm{x}} \\&+\frac{16}{5}\Big\|T(t)\Big\|^{5}_{L^{5}(\Omega)}+C(a,\Omega)\int_{\partial\Omega}g^{5}(t,{\bm{x}})d{\bm{\Gamma}},
\end{aligned}
\label{Newist}
\end{equation}
using the Young inequality we obtain 
\begin{equation}
\Big\|T(t)\Big\|^{5}_{L^{5}(\Omega)} \leqslant \frac{5 \epsilon}{8}\Big\|T(t)\Big\|^{8}_{L^{8}(\Omega)}+ C(\epsilon)\Big|\Omega\Big|.
\label{estimMichel}
\end{equation}
Integrating \eqref{Newist} in time between $0$ and ${\bm{\tau}}$ and using \eqref{estimMichel} we deduce
\begin{equation*}
\begin{aligned}
\Big\|T({\bm{\tau}})\Big\|^{5}_{L^{5}(\Omega)}+&10\pi \theta \Big\|T\Big\|^{8}_{L^{8}({\bm{Q_{\tau}}})} \leqslant \frac{5\theta}{8\pi}\Big\|G\Big\|^{2}_{L^{2}({\bm{Q_{\tau}}})}+2\epsilon\Big\|T\Big\|^{8}_{L^{8}({\bm{Q_{\tau}}})}\\&+C(\epsilon){\bm{\tau}}\Big|\Omega\Big|+C(a,\Omega)\int_{{\bm{\Sigma_{{\bm{\tau}}}}}}g^{5}(s,{\bm{x}})d{\bm{\Gamma}} ds \\&+\Big\|T_{0}\Big\|^{5}_{L^{5}(\Omega)}.
\end{aligned}
\label{New1}
\end{equation*}
Taking  ${\displaystyle \epsilon =\pi \theta}$, it follows that
\begin{equation*}
\begin{aligned}
 \Big\|T\Big\|^{8}_{L^{8}({\bm{Q_{\tau}}})} \leqslant &\frac{5}{64\pi^{2}}\Big\|G\Big\|^{2}_{L^{2}({\bm{Q_{\tau}}})}+{\bm{\tau}}C(\theta)\Big|\Omega\Big|+C(a,\Omega,\theta)\int_{{\bm{\Sigma_{{\bm{\tau}}}}}}g^{5}(s,{\bm{x}})d{\bm{\Gamma}} ds\\&+\frac{1}{8\pi \theta}\Big\|T_{0}\Big\|^{5}_{L^{5}(\Omega)},
 \end{aligned}
\end{equation*}
then
\begin{equation*}
\begin{aligned}
 \Big\|T\Big\|^{8}_{L^{8}({\bm{Q_{\tau}}})} \leqslant& \frac{5}{64\pi^{2}}\Big\|G\Big\|^{2}_{L^{2}({\bm{Q_{\tau}}})}+{\bm{\tau}}\left(C(\theta)\Big|\Omega\Big|+C(a,\Omega,\theta)\Big\|g\Big\|^{5}_{L^{\infty}((0,{\bm{\tau}})\times\overline{\Omega})}\right)\\&+\frac{1}{8\pi \theta}\Big\|T_{0}\Big\|^{5}_{L^{5}(\Omega)}.
\end{aligned}
\label{New1}
\end{equation*}
Finally, we consider the case of Dirichlet boundary conditions ${(a=0, b\neq0)}$. This type of boundary condition requests a different analytical tool.

To bound the last term on the right hand side of (\ref{CBgeneral}), we multiply equation (\ref{chaleur}) by $g^{4}$ (given in (\ref{assumptions})) and we integrate over ${\bm{Q_{\tau}}}$, we get 
\begin{equation*}
\begin{aligned}
\int_{0}^{{\bm{\tau}}}\int_{\Omega}\Big[\partial_{t}T(s,{\bm{x}}) -\Delta T(s,{\bm{x}}) +4\pi \theta T^{4}(s,{\bm{x}}) \Big] &g^{4}(s,{\bm{x}}) ds d{\bm{x}}\\&=\int_{0}^{{\bm{\tau}}}\int_{\Omega}G(s,{\bm{x}}) g^{4}(s,{\bm{x}}) ds d{\bm{x}}.
\end{aligned}
\end{equation*}
Therefore, we deduce from Green's Formula that 
\begin{equation}
\begin{aligned}
\int_{\Omega}\Big[&T({\bm{\tau}},{\bm{x}})g^{4}(t,{\bm{x}})-T(0,{\bm{x}}) g^{4}(0,{\bm{x}}) \Big]d{\bm{x}}-\int_{\Omega}\int_{0}^{{\bm{\tau}}}T(s,{\bm{x}})\partial_{t}g^{4}(s,{\bm{x}}) ds d{\bm{x}}\\&-\int_{{\bm{Q_{\tau}}}}T(s,{\bm{x}}) \Delta(g^{4})(s,{\bm{x}})ds d{\bm{x}}+\frac{1}{b}\int_{{\bm{\Sigma_{{\bm{\tau}}}}}}g(s,{\bm{x}}) \partial_{n}g^{4}(s,{\bm{x}}) ds d{\bm{\Gamma}}\\&-\int_{{\bm{\Sigma_{{\bm{\tau}}}}}}\partial_{n}T(s,{\bm{x}}) g^{4}(s,{\bm{x}}) ds d{\bm{\Gamma}}+4\pi \theta\int_{{\bm{Q_{\tau}}}}T^{4}(s,{\bm{x}}) g^{4}(s,{\bm{x}}) ds d{\bm{x}}\\&=\int_{{\bm{Q_{\tau}}}}G(s,{\bm{x}}) g^{4}(s,{\bm{x}})ds d{\bm{x}}.
\end{aligned}
\label{Aux4t}
\end{equation}
Using  the positivity of $G$ and $T_{0}$, (\ref{Aux4t}) becomes
\begin{equation*}
\begin{aligned}
\int_{{\bm{\Sigma_{{\bm{\tau}}}}}}\partial_{n}T(s,{\bm{x}})&g^{4}(s)ds d{\bm{x}}\leqslant4\pi \theta\int_{{\bm{Q_{\tau}}}} T^{4}(s,{\bm{x}}) g^{4}(s,{\bm{x}}) ds d{\bm{x}}+\int_{\Omega}T({\bm{\tau}},x)g^{4}({\bm{\tau}},{\bm{x}})d{\bm{x}}\\&-\int_{{\bm{Q_{\tau}}}}T(s,{\bm{x}})\left(\partial_{t}-\Delta\right)g^{4}(s,{\bm{x}})ds d{\bm{x}}\\\
&+\frac{1}{b}\int_{{\bm{\Sigma_{{\bm{\tau}}}}}}g(s,{\bm{x}})\partial_{n}g^{4}(s,{\bm{x}})ds d d{\bm{\Gamma}}.
\end{aligned}
\end{equation*}
Then, Young's inequality implies that there exist $\epsilon_{1}>0$, $\epsilon_{2}>0$, $\epsilon_{3}>0$
\begin{equation}
\begin{aligned}
\int_{{\bm{\Sigma_{{\bm{\tau}}}}}}\partial_{n}T(s,{\bm{x}})&g^{4}(s,{\bm{x}}) ds d{\bm{x}}\leqslant\frac{2\pi \theta}{\epsilon_{1}}\Big\|T\Big\|^{8}_{L^{8}({\bm{Q_{\tau}}})}+2\pi \theta {\bm{\tau}}C(\epsilon_{1})\Big|\Omega\Big|\Big\|g\Big\|^{8}_{L^{\infty}({\bm{Q_{\tau}}})}\\&+\frac{1}{8\epsilon_{2}}\Big\|T\Big\|^{8}_{L^{8}({\bm{Q_{\tau}}})}+C(\epsilon_{2}){\bm{\tau}}\Big|\Omega\Big|\Big\|\left(\partial_{t}-\Delta\right)g^{4}\Big\|^{8/7 }_{L^{\infty}({\bm{Q_{\tau}}})}\\&+\frac{1}{5\epsilon_{3}}\Big\|T({\bm{\tau}})\Big\|^{5}_{L^{5}(\Omega)}+C(\epsilon_{3})\Big|\Omega\Big|\Big\|g\Big\|^{5}_{L^{\infty}({\bm{Q_{\tau}}})}\\&+\frac{{\bm{\tau}}}{b}\Big|\partial\Omega\Big| \Big\|g\partial_{n}g^{4}\Big\|_{L^{\infty}((0,{\bm{\tau}})\times\overline{\Omega})}.
\end{aligned}
\label{CBDirch}
\end{equation}
For the Dirichlet boundary conditions, the inequality \eqref{CBgeneral} becomes
\begin{equation}
\begin{aligned}
 \frac{1}{5} \frac{d}{dt} \Big\|T(t )\Big\|^{5}_{L^{5}(\Omega)} +&\frac{16}{25}\int_{\Omega}(\nabla T^{\frac{5}{2}}(t,{\bm{x}}))^{2}d{\bm{x}}+2\pi \theta \int_{\Omega}T^{8}(t,{\bm{x}})d{\bm{x}}\\&\leqslant \frac{\theta}{8\pi}\int_{\Omega}G^{2}(t,{\bm{x}})d{\bm{x}}+\frac{1}{b^{4}}\int_{\partial\Omega}\partial_{n} T(t,{\bm{x}})g^{4}(t,{\bm{x}})d{\bm{\Gamma}}.
\end{aligned}
\label{CBgeneraldir}
\end{equation}
Integrating (\ref{CBgeneraldir}) in time between $0$ and ${\bm{\tau}}$ and using \eqref{CBDirch} we obtain
\begin{equation*}
\begin{aligned}
\Big\|T({\bm{\tau}})\Big\|^{5}_{L^{5}(\Omega)}+&10\pi \theta \Big\|T\Big\|^{8}_{L^{8}({\bm{Q_{\tau}}})} \leqslant\frac{5\theta}{8\pi}\Big\|G\Big\|^{2}_{L^{2}({\bm{Q_{\tau}}})}+\frac{5\pi \theta}{b^{4}\epsilon_{1}}\Big\|T\Big\|^{8}_{L^{8}({\bm{Q_{\tau}}})}\\&+\frac{5}{b^{4}}\pi \theta {\bm{\tau}}C(\epsilon_{1})\Big|\Omega\Big|\Big\|g\Big\|^{8}_{L^{\infty}({\bm{Q_{\tau}}})}+\frac{5}{8b^{4}\epsilon_{2}}\Big\|T\Big\|^{8}_{L^{8}({\bm{Q_{\tau}}})}\\&+C(b,\epsilon_{2}){\bm{\tau}}\Big|\Omega\Big|\Big\|\left(\partial_{t}-\Delta\right)g^{4}\Big\|^{8/7 }_{L^{\infty}({\bm{Q_{\tau}}})}\\&+\frac{1}{b^{4}\epsilon_{3}}\Big\|T({\bm{\tau}})\Big\|^{5}_{L^{5}(\Omega)}+C(\epsilon_{3})\Big|\Omega\Big|\Big\|g\Big\|^{5}_{L^{\infty}({\bm{Q_{\tau}}})}\\&+\frac{5{\bm{\tau}}}{b^{5}}\Big|\partial\Omega\Big| \Big\|g\partial_{n}g^{4}\Big\|_{L^{\infty}((0,{\bm{\tau}})\times\overline{\Omega})}\\&+\Big\|T_{0}\Big\|^{5}_{L^{5}(\Omega)}.
\end{aligned}
\end{equation*}
Choosing ${\displaystyle \epsilon_{1}=\frac{5}{2 b^{4}},\epsilon_{2}=\frac{5}{24\pi \theta b^{4}}\, \text{and}\,\epsilon_{3}=\frac{2}{ b^{4}}}$, then
\begin{equation*}
\begin{aligned}
\Big\|T\Big\|^{8}_{L^{8}({\bm{Q_{\tau}}})} \leqslant&\frac{1}{8\pi^{2}}\Big\|G\Big\|^{2}_{L^{2}({\bm{Q_{\tau}}})}+C(b) {\bm{\tau}}\Big|\Omega\Big|\Big\|g\Big\|^{8}_{L^{\infty}({\bm{Q_{\tau}}})}+\frac{1}{5\pi \theta}\Big\|T_{0}\Big\|^{5}_{L^{5}(\Omega)}\\&+C(\theta,b)\Big|\Omega\Big|\Big\|g\Big\|^{5}_{L^{\infty}({\bm{Q_{\tau}}})}+C(\theta,b){\bm{\tau}}\Big|\partial\Omega\Big| \Big\|g\partial_{n}g^{4}\Big\|_{L^{\infty}((0,{\bm{\tau}})\times\overline{\Omega})}\\&+C(\theta,b){\bm{\tau}}\Big|\Omega\Big|\Big\|\left(\partial_{t}-\Delta\right)g^{4}\Big\|^{8/7 }_{L^{\infty}({\bm{Q_{\tau}}})}.
\end{aligned}
\end{equation*}
 This finishes the proof.
  \end{proof}
 \begin{remark}
Animmediate consequence of Propositions \ref{propchalex} and \ref{thmchalex} is that  we can reduce the regularity of $g$, it suffices to take ${g\in H^{1/4}(0,{\bm{\tau}};L^{2}(\partial \Omega))}$ ${\cap}$ ${L^{\infty}({\bm{\Sigma_{{\bm{\tau}}}}})}$ ${\cap}$  ${L^{2}(0,{\bm{\tau}};H^{1/2}(\Omega))}$ for Robin and Neumann boundary conditions case and take ${g\in H^{3/4}(0,{\bm{\tau}};L^{2}(\partial \Omega))}$ ${\cap}$ ${L^{\infty}({\bm{\Sigma_{{\bm{\tau}}}}})}$ ${\cap}$  ${L^{2}(0,{\bm{\tau}};H^{3/2}(\Omega))}$ for Dirichlet boundary conditions. For more information on the regularity of the trace operator we refer the reader to \cite{Denk2007} and \cite{Pouso2000}.
 \end{remark}
\subsection{Existence and uniqueness of the solution}
Now, we may give a very direct proof of theorem \ref{thmintro} using Banach fixed point theorem.
\begin{proof}[Proof of theorem \ref{thmintro}]

 $\mathcal{H}=\mathcal{H}_{3}\circ \mathcal{H}_{2} \circ \mathcal{H}_{1}$  is a well posed  continuous  map from $W_{2}^{2,1}({\bm{Q_{\tau}}})$ to $W_{2}^{2,1}({\bm{Q_{\tau}}})$ because it is composed of a three well posed continuous map. Moreover, from \eqref{G} and \eqref{C3} there exists $C_{4}=C({\bm{\tau}},\Omega,\theta)>0$ such that 

\begin{equation}
\Big\|\mathcal{H}(T)\Big\|_{W_{2}^{2,1}({\bm{Q_{\tau}}})}\leqslant C_{4}\Big(\Big\|T \Big\|^{4}_{W_{2}^{2,1}({\bm{Q_{\tau}}})}+\Big\|T_{0}\Big\|_{H^{1}(\Omega)}+\Big\| g \Big\|_{L^{2}(0,{\bm{\tau}};H^{\frac{3}{2}}(\Omega))}+\Big\|I_{b}\Big\|_{L^{2}(0,{\bm{\tau}};L^{2}_{-})}\Big).
\label{LYInj1x}
\end{equation}

 Now, we prove the existence and uniqueness of the solution for the coupled system (\ref{radiatif2})-(\ref{chaleur5}). Let us consider  $(T_{1},T_{2}) \in {W_{2}^{2,1}({\bm{Q_{\tau}}})}^{2}$. From \eqref{G}, \eqref{C3} and \eqref{LYInj1x} we have 
 \begin{equation*}
\Big\|\mathcal{H}(T_{1})-\mathcal{H}(T_{2})\Big\|_{W_{2}^{2,1}({\bm{Q_{\tau}}})}\leqslant C_{4}\Big\|T_{1}^{4}-T_{2}^{4}\Big\|_{L^{2}({\bm{Q_{{\bm{\tau}}}}})}.
\end{equation*}

We denote by $B(0,r)$ the closed ball of radius $r$ in $W_{2}^{2,1}({\bm{Q_{\tau}}})$ with $r$ satisfies ${r^{3}<\frac{1}{4C^{\ast}_{4}}}$, where  $C^{\ast}_{4}=\max_{}\{C_{4},C_{4}C_{5}\}$, $C_{5}$ is given as follows. Using the generalized H\"{o}lder's inequality, we have the following inequality
 \begin{equation}
\Big\|T_{1}^{4}-T_{2}^{4}\Big\|^{2}_{L^{2}({\bm{Q_{{\bm{\tau}}}}})}\leqslant \Big\|T_{1}-T_{2}\Big\|_{L^{8}({\bm{Q_{{\bm{\tau}}}}})}^{2}\Big\|T_{1}+T_{2}\Big\|_{L^{8}({\bm{Q_{{\bm{\tau}}}}})}^{2} \Big\|T_{1}^{2}+T_{2}^{2}\Big\|_{L^{4}({\bm{Q_{{\bm{\tau}}}}})}^{2}.
\label{estimatinSemi}
 \end{equation}
We assume that  $T_{1},T_{2} \in B(0,r)$ then there exists $C_{5}=C({\bm \tau},{\bm\Omega})>0$ such that
  \begin{equation*}
  \Big\|T_{1}^{4}-T_{2}^{4}\Big\|_{L^{2}({\bm{Q_{{\bm{\tau}}}}})}\leqslant 4 C_{5}r^{3}\Big\|T_{1}-T_{2}\Big\|_{W_{2}^{2,1}({\bm{Q_{\tau}}})}.
 \end{equation*}
Then 
\begin{equation*}
\Big\|\mathcal{H}(T_{2})-\mathcal{H}(T_{1})\Big\|_{W_{2}^{2,1}({\bm{Q_{\tau}}})}\leqslant 4 C_{4}C_{5}r^{3}\Big\|T_{1}-T_{2}\Big\|_{W_{2}^{2,1}({\bm{Q_{\tau}}})}.
\end{equation*}
Let us assume that 

$$\Big\|T_{0}\Big\|_{H^{1}(\Omega)}\leqslant \frac{r}{4C^{\ast}_{4}}$$
$$\Big\| g \Big\|_{L^{2}(0,{\bm{\tau}};H^{\frac{3}{2}}(\Omega))}\leqslant \frac{r}{4C^{\ast}_{4}}$$
$$\Big\|I_{b}\Big\|_{L^{2}(0,{\bm{\tau}};L^{2}_{-})}\leqslant \frac{r}{4C^{\ast}_{4}}.$$

Thus $\mathcal{H}(B(0,r))\subset B(0,r)$, then we deduce that $\mathcal{H}$ is a contraction map from $B(0,r)$ to $B(0,r)$. Finally, $\mathcal{H}$ admits a unique fixed point$T \in B(0,r)$ such that  $\mathcal{H}(T)=T$. This implies  the existence and uniqueness of the solution in $W_{2}^{2,1}({\bm{Q_{\tau}}})$. Therefore, by theorem \ref{thsolrad} and Proposition \ref{propchalex} the system (\ref{radiatif2})-(\ref{chaleur5}) has a unique solution $(T,I)\in W_{2}^{2,1}({\bm{Q_{\tau}}})\times L^{2}(0,{\bm{\tau}};{\bm{\mathcal{W}}})$.
\end{proof}

{\bf{Acknowledgements: }}
The authors would like to thank Professor {\bf Michel PIERRE} (\'Ecole Normale Sup\'erieure de Rennes, IRMAR, France) for many helpful discussions and comments. Moreover, we want to thank the anonymous {\bf{referees}} for their suggestions which led to a improvement of the original manuscript.

\end{document}